\documentclass[12pt,leqno,amscd, amsfonts, amssymb,pstricks,verbatim]{amsart}
\usepackage{mathrsfs}
\usepackage{amsmath}
\usepackage{hyperref}
\usepackage{bookmark}
\usepackage{cite}

\def\b{\beta}
\def\d{\Delta}

\def\Z{\mathbb{Z}}
\def\N{\mathbb{N}}

\def\C{\mathbb{C}}

\numberwithin{equation}{section}
\newtheorem{theo}{Theorem}[section]

\newtheorem{lemm}[theo]{Lemma}

\allowdisplaybreaks

\begin{document}
	
	\title[2-Local derivations on a Block-type Lie algebra]{2-Local derivations on a Block-type Lie algebra}
	
\author{Qiufan Chen}

\address{Chen: Department of Mathematics, Shanghai Maritime University,
 Shanghai, 201306, China.}\email{chenqf@shmtu.edu.cn}	
	
\author{Xiaohan Guo}

\address{Guo: Department of Mathematics, Shanghai Maritime University,
 Shanghai, 201306, China.}\email{202531010005@stu.shmtu.edu.cn}

	
	
\thanks{This work is supported by National Natural Science Foundation of China (Grant Nos. 12271345 and 12361006) and Science and Technology Commission of Shanghai Municipality (Grant No. 25ZR1402183). }
	
\begin{abstract}
The present paper is devoted to study 2-local derivations on the Block-type Lie algebra which is an infinite-dimensional Lie algebra with some outer derivations. We  prove that
every 2-local derivation on the Block-type Lie algebra is a derivation.
\end{abstract}
	
	\maketitle
	\setcounter{tocdepth}{1}\tableofcontents
	\begin{center}
	\end{center}
	
	\section{Introduction}
Throughout the paper, we denote by $\C ,\,\Z,\,\mathbb{Z}_{+},\N$ the sets of complex numbers, integers, nonnegative integers and positive integers, respectively. All algebras and vector spaces considered in the paper are over $\C$.

As a generalization of derivation, $\check{\text{S}}$emrl introduced the notion of 2-local derivation in \cite{Se}. Therein, he showed that every 2-local derivation on $B(H)$, the algebra  of all bounded linear operators on the infinite-dimensional separable Hilbert space $H$, is a derivation. The idea was then exploited and generalized to consider 2-local derivations on different classes of associative algebras, such as von Neumann algebras and matrix algebras over commutative regular algebras \cite{AK,AKA}. The concept of 2-local derivation is indeed important and interesting for an algebra. The main problem in this subject is to determine all 2-local derivations and examine whether they are necessarily (global) derivations.  For the Lie (super) algebra case, in \cite{AKR}, it was shown that each 2-local derivation on a finite-dimensional semisimple Lie algebra over an algebraically closed field of characteristic zero is a derivation and each finite-dimensional nilpotent Lie algebra with dimension larger than two admits a 2-local derivation which is not a derivation. The authors in \cite{WCN} proved that all 2-local superderivations on basic classical Lie superalgebras except $A(n,n)$ over an algebraically closed field of characteristic zero are derivations. In recent years, many scholars have paid much attention to the study of 2-local (super)derivations of some other Lie (super) algebras, see \cite{A,AKY,CH,DGL,SYZ,TT,TXW,T,YZ,ZCZ}. There is no uniform method to determine all 2-local (super) derivations on Lie (super) algebras.

Block-type Lie algebras  were first introduced by Block \cite{B} in 1958. As is well-known, the Lie algebras of this type have close relations with various well-known Lie algebras, such as the Virasoro algebra \cite{M}, the Virasoro-like algebra \cite{ZZ}, (generalized) Cartan type Lie algebras \cite{X}, W-infinity algebras \cite{KR} and have some applications in the integrable systems \cite{LHS,LHS2}.  Due to these reasons, the structure theory and representation theory of Block-type Lie algebras have been intensively studied by many authors, see \cite{DZ,OZ,S,SXX,SY}. 
However,  because of their inherent complexity, the structure theory of Block-type Lie algebras, especially their local properties, has so far received insufficient attention in the literature.
The aim of this paper is to study 2-local derivations on the Block-type Lie algebra, denoted by  $\mathcal{B}$,  with  basis  \(\{L_{\alpha, i} \mid \alpha \in \mathbb{Z}, i \in \mathbb{Z}_{+}\}\) and  relations
    \[
	 [L_{\alpha, i}, L_{\beta, j} ]=((\alpha-1)(j+1)-(\beta-1)(i+1)) L_{\alpha+\beta, i+j},\ \ \ \ \alpha, \beta \in \mathbb{Z}, i, j \in \mathbb{Z}_{+}.
	\]
This special Block-type Lie algebra is interesting in the sense that the central extension of \(\mathcal{B}\) is given by 
	\[
	 [L_{\alpha, i}, L_{\beta, j} ]=((\alpha-1)(j+1)-(\beta-1)(i+1)) L_{\alpha+\beta, i+j}+\delta_{\alpha+\beta, 0} \delta_{i, 0} \delta_{j, 0} \frac{\alpha^{3}-\alpha}{6} C,\ \ \ \ \alpha, \beta \in \mathbb{Z}, i, j \in \mathbb{Z}_{+},
	\]
which contains a subalgebra  spanned by \(\{L_{\alpha, 0}, C \mid \alpha \in \mathbb{Z}\}\) that is isomorphic to the  Virasoro algebra. Derivation algebra, automorphism group and quasifinite representations of $\mathcal{B}$ have been considered in \cite{S,XW}. 
%From \cite{XW}, we know that  $\mathcal{B}$ has some outer derivations. So it takes many difficulties to determine its 2-local derivations, which can be seen from the computing complexity in this paper.
	
The paper is organized as follows. In Section 2, we recall some fundamental definitions and basic known results that we need in the following. In Section 3, we prove that every 2-local derivation on the Block-type Lie algebra  $\mathcal{B}$ is automatically a derivation.
%Throughout the paper, we denote by $\Z ,\,\Z^*,\,\Z_+,\,\N,\,\C$ the sets of integers, nonzero integers, non-negative integers, positive integers and complex numbers, respectively. All algebras considered in the paper are over $\C$.
\section{Preliminaries}
In this section, we give some necessary definitions and preliminary results.
	
A {\bf derivation} on a Lie algebra $\mathcal{L}$  is a linear map $D:\mathcal{L}\rightarrow\mathcal{L}$ such that
	$$D([x, y])=[D(x), y]+[x, D(y)],\,\,\forall\,x,y\in\mathcal{L}.$$
The set of all derivations of $\mathcal{L}$ is denoted by $\mathrm{Der}(\mathcal{L})$. For each element $a\in\mathcal{L}$, the map 
	$$\mathrm{ad} (a):\mathcal{L}\rightarrow\mathcal{L},\,\,\mathrm{ad} (a) (x)=[a, x], \,\,\forall\,x\in\mathcal{L}$$ is a derivation of $\mathcal{L}$ and derivations of this form are called 
{\bf inner derivations}. The set of all inner derivations, denoted by $\mathrm{Inn}(\mathcal{L})$, is an ideal of $\mathrm{Der}(\mathcal{L})$.
	
A map $\Delta: \mathcal{L}\rightarrow\mathcal{L}$ (not necessarily linear) is called  a {\bf 2-local derivation} if for any $x,y\in\mathcal{L}$, there exists a derivation $D_{x,y}\in\mathrm{Der}(\mathcal{L})$ (depending on $x, y$) such that $\Delta(x)=D_{x,y}(x)$ and $\Delta(y)=D_{x,y}(y)$.
In particular, for any $x\in\mathcal{L}$ and $k\in\C$, there exists $D_{kx,x}\in\mathrm{Der}(\mathcal{L})$ such that
$$\Delta(kx)=D_{kx,x}(kx)=kD_{kx,x}(x)=k\Delta(x).$$
	
We now recall and establish several auxiliary results.
	\begin{lemm}\label{prop1}(cf. \cite{XW})
		$\mathrm{Der}\mathcal{B}=\mathrm{ad}\mathcal{B}\oplus\C d$, where
		$d$ is an outer derivation defined by $d(L_{\b, j})=\b L_{\b, j}$ for all $ \b\in\Z,j\in\Z_+$.
	\end{lemm}
	As a direct consequence of Lemma \ref{prop1}, we have the following.
	\begin{lemm}\label{prop2}
	Let $\Delta$ be a 2-local derivation on $\mathcal{B}$. Then for every $x,y\in \mathcal{B}$, there exists a derivation  $D_{x,y}$ of $\mathcal{B}$ such that  $\Delta(x)=D_{x,y}(x),\Delta(y)=D_{x,y}(y)$ and  it can be written as
	$$D_{x, y}=\mathrm{ad}(\sum_{(\alpha,i)\in \mathbb{Z}\times\mathbb{Z}_{+}} a_{\alpha, i}(x, y) L_{\alpha, i})+\lambda(x, y) d,$$
	where $a_{\alpha, i}(x, y),\lambda(x, y)\in\C$  and $d$ is given by Lemma \ref{prop1}.

	\end{lemm}
	\section{ 2-local derivations on $\mathcal{B}$}
In this section, we will determine all 2-local derivations on the Block-type Lie algebra $\mathcal{B}$.
	\begin{lemm}\label{prop3}
		Let $\Delta$ be a 2-local derivation on $\mathcal{B}$, and let $(\b,j)\in\Z\times \Z_+$ such that  $\Delta(L_{\b,j})=0$. Then for any $y\in\mathcal{B}$, we have 
\begin{align}
\label{e1}D_{L_{\beta, j}, y}&=\mathrm{ad}(\frac{\beta}{\beta+j} \lambda (L_{\beta, j}, y) L_{0,0}+\sum_{i \in \Z_{+}} a_{k_{i}, i} (L_{\beta, j}, y) L_{k_{i}, i})+\lambda (L_{\beta, j}, y) d\,\,\,{\rm for}\,\,\,j\neq -\b,\\
\label{e2}D_{L_{\beta,-\beta}, y}&=\mathrm{ad}(\sum_{i \in \Z_{+}} a_{-i, i}(L_{\beta,-\beta}, y) L_{-i, i})+\delta_{\beta, 0} \lambda(L_{0,0}, y)d,
			\end{align}
where $k_{i}=\frac{\beta i+\b-i+j}{j+1} \in \mathbb{Z}, a_{k_{i},i} (L_{\beta, j}, y),a_{-i, i}(L_{\beta,-\beta}, y),\lambda (L_{\beta, j}, y),\lambda(L_{0,0}, y) \in \mathbb{C}$ and $\delta_{\beta, 0}$ is the Kronecker delta.
	\end{lemm}
	\begin{proof}
		Using Lemma \ref{prop2}, for any $j\in\Z_+$, we can assume that
	    \[
	    D_{L_{\beta, j}, y}=\mathrm{ad}(\sum_{(\alpha,i)\in \mathbb{Z}\times\mathbb{Z}_{+}} a_{\alpha, i}(L_{\beta, j}, y) L_{\alpha, i})+\lambda(L_{\beta, j}, y) d
	    \]
with $a_{\alpha, i}(L_{\beta, j}, y),\lambda(L_{\beta, j}, y)\in\mathbb{C}$. First consider the case $j\neq-\b$. In this case, we have
\begin{align*}
			\Delta(L_{\beta, j}) & =D_{L_{\beta, j}, y}(L_{\beta, j}) \\
			& =\sum_{(\alpha,i)\in \mathbb{Z}\times\mathbb{Z}_{+}} [a_{\alpha, i} (L_{\beta, j}, y) L_{\alpha, i}, L_{\beta, j}]+\lambda (L_{\beta, j}, y) d (L_{\beta, j}) \\
			& =\sum_{(\alpha,i)\in \mathbb{Z}\times\mathbb{Z}_{+}}((\alpha-1)(j+1)-(\beta-1)(i+1)) a_{\alpha, i}(L_{\beta, j}, y) L_{\alpha+\beta, i+j}+\beta \lambda(L_{\beta, j}, y) L_{\beta, j} \\
			& =0.
		\end{align*}
For $(\alpha,i)\in \mathbb{Z}\times\mathbb{Z}_{+}\setminus\{(0,0)\}$, by comparing the coefficients of $L_{\alpha+\beta, i+j}$  and   $L_{\beta, j}$ in the above equation, we respectively get
\[((\alpha-1)(j+1)-(\beta-1)(i+1)) a_{\alpha, i} (L_{\beta, j}, y)=0\]
and
		\[-(\beta+j) a_{0,0} (L_{\beta, j}, y)+\beta \lambda (L_{\beta, j}, y)=0.\]	
The first equation implies that $a_{\alpha, i}(L_{\beta, j}, y)=0$ provided that $\alpha \neq \frac{\beta i+\b-i+j}{j+1}$. From the second equation, we obtain $a_{0,0} (L_{\beta, j}, y )=\frac{\beta}{\beta+j} \lambda(L_{\beta, j}, y)$. Thus, \eqref{e1} holds. In the remaining case $j=-\b$,  a straightforward computation yields that
	\begin{align*}
		\Delta (L_{\beta, -\b} ) & =D_{L_{\beta, -\b}, y} (L_{\beta, -\b}) \\
		& = \sum_{(\alpha,i)\in \mathbb{Z}\times\mathbb{Z}_{+}} [a_{\alpha, i} (L_{\beta,-\beta}, y) L_{\alpha, i}, L_{\beta,-\beta} ]+\lambda (L_{\beta,-\beta}, y) d (L_{\beta,-\beta}) \\
		& =\sum_{(\alpha,i)\in \mathbb{Z}\times\mathbb{Z}_{+}}(1-\beta)(\alpha+i) a_{\alpha, i} (L_{\beta,-\beta}, y) L_{\alpha+\beta, i-\beta}+\beta \lambda (L_{\beta,-\beta}, y) L_{\beta,-\beta}\\
		& =0.
		\end{align*}
For $(\alpha,i)\in \mathbb{Z}\times\mathbb{Z}_{+}\setminus\{(0,0)\}$, it follows from comparing the coefficients of $L_{\beta,-\beta}$ and $L_{\alpha+\beta, i-\beta}$ in the above equation that 
	\[
	\begin{gathered}
		\beta \lambda (L_{\beta,-\beta}, y)=0, \\
		(1-\beta)(\alpha+i) a_{\alpha, i} (L_{\beta,-\beta}, y)=0,
	\end{gathered}
	\]
which in turn imply $\lambda(L_{\beta,-\beta}, y)=0$ for $\beta \neq 0$ and $a_{\alpha, i} (L_{\beta,-\beta}, y )=0$ for $\alpha \neq-i$. This proves \eqref{e2}, and with it the lemma.
\end{proof}
	\begin{lemm}\label{prop4}
		Let $\Delta$ be a 2-local derivation on $\mathcal{B}$ such that $\Delta(L_{0,0})=\Delta(L_{1,0})=0$. Then for any $(\beta, j)\in \mathbb{Z}\times\mathbb{Z}_{+}$, we have
$\Delta (L_{\beta, j} )=j \xi_{L_{\beta, j}} L_{\beta, j}$, where $\xi_{L_{\beta, j}}=-\lambda(L_{1,0}, L_{\beta, j})$.
	\end{lemm}
	\begin{proof} According to \(\Delta (L_{0,0} )=\Delta (L_{1,0} )=0\) and Lemma \ref{prop3}, for  any $(\beta, j)\in \mathbb{Z}\times\mathbb{Z}_{+}$,  we have
\begin{align*}
			D_{L_{1,0}, L_{\beta, j}}&=\mathrm{ad} (\lambda (L_{1,0}, L_{\beta, j}) L_{0,0}+\sum_{i \in \Z_{+}} a_{1, i} (L_{1,0}, y) L_{1, i})+\lambda(L_{1,0}, L_{\beta, j}) d,\\
			D_{L_{0,0}, L_{\beta, j}}&=\mathrm{ad} (\sum_{i \in \Z_{+}} a_{-i, i}(L_{0,0}, L_{\beta, j}) L_{-i, i} )+\lambda (L_{0,0}, L_{\beta, j}) d
		\end{align*}
for some \(\lambda (L_{1,0},L_{\beta, j}), a_{1, i} (L_{1,0},L_{\beta, j}),a_{-i, i} (L_{0,0},L_{\beta, j} ),\lambda (L_{0,0}, L_{\beta, j})\in \mathbb{C}\).
We now calculate explicitly
\begin{align*}
\Delta (L_{\beta, j})= & D_{L_{1,0}, L_{\beta, j}} (L_{\beta, j}) \\
		= & { [\lambda\left(L_{1,0}, L_{\beta, j}\right) L_{0,0}, L_{\beta, j}]+ \sum_{i \in \mathbb{Z}_{+}}[a_{1, i} (L_{1,0}, L_{\beta, j}) L_{1, i}, L_{\beta, j}]+\lambda\left(L_{1,0}, L_{\beta, j}\right) d\left(L_{\beta, j}\right)} \\
%		= & -(\beta+j) \lambda (L_{1,0}, L_{\beta, j} ) L_{\beta, j}+\sum_{i \in \mathbb{Z}_{+}}-(\beta-1)(i+1) a_{1, i} (L_{1,0}, L_{\beta, j} ) L_{1+\beta, i+j} \\
%		& +\beta \lambda (L_{1,0}, L_{\beta, j} ) L_{\beta, j} \\
		= & -\sum_{i \in \mathbb{Z}_{+}}(\beta-1)(i+1) a_{1, i} (L_{1,0}, L_{\beta, j}) L_{1+\beta, i+j}-j \lambda (L_{1,0}, L_{\beta, j}) L_{\beta, j},\\
\Delta (L_{\beta, j})= & D_{L_{0,0}, L_{\beta, j}} (L_{\beta, j} ) \\
		= & { \sum_{i \in \mathbb{Z}_{+}} [a_{-i, i}\left(L_{0,0}, L_{\beta, j}\right) L_{-i, i}, L_{\beta, j} ]+\lambda (L_{0,0}, L_{\beta, j}) d(L_{\beta, j})} \\
		= &-\sum_{i \in \mathbb{Z}_{+}}(i+1)(\beta+j) a_{-i, i} (L_{0,0}, L_{\beta, j} ) L_{-i+\beta, i+j}+\beta \lambda (L_{0,0}, L_{\beta, j}) L_{\beta, j}.
		\end{align*}
Equating the above two expressions for $\Delta (L_{\beta, j})$ and comparing the coefficients of \(L_{1+\beta, i+j}\) on both sides, we immediately get $(\beta-1)(i+1) a_{1, i} (L_{1,0}, L_{\beta, j})=0$ for all $i\in\Z_+$.
%		\begin{align*}
%			& -\sum_{i \in \mathbb{Z}_{+}}(\beta-1)(i+1) a_{1, i}(L_{1,0}, L_{\beta, j}) L_{1+\beta, i+j}-j \lambda (L_{1,0}, L_{\beta, j}) L_{\beta, j}\\
%			& =-\sum_{i \in \mathbb{Z}_{+}}(i+1)(\beta+j) a_{-i, i} (L_{0,0}, L_{\beta, j}) L_{-i+\beta, i+j}+\beta \lambda (L_{0,0}, L_{\beta, j}) L_{\beta, j}.
%		\end{align*}
%Comparing the coefficients of  in the above equation gives 
As a result, \(\Delta\left(L_{\beta, j}\right)=j \xi_{L_{\beta, j}} L_{\beta, j}\) with \(\xi_{L_{\beta, j}}=- \lambda(L_{1,0}, L_{\beta, j}) \in \mathbb{C}\), as desired.
	\end{proof}
	\begin{lemm}\label{prop5}
		Let $\Delta$ be a 2-local derivation on $\mathcal{B}$ such that \(\Delta (L_{0,0})=\Delta (L_{1,0})=0\). Then for any \(x=\sum_{(\gamma,k) \in \mathbb{Z}\times \mathbb{Z}_{+}} \mu_{\gamma, k} L_{\gamma, k}\), we have $\Delta(x)=\xi_{x} \sum_{(\gamma,k) \in \mathbb{Z}\times \mathbb{Z}_{+}} k\mu_{\gamma, k} L_{\gamma, k}$, where  $\xi_{x}=-\lambda (L_{n, 0}, x) \in \mathbb{C}$.
	\end{lemm}
	\begin{proof}
Combining \(\Delta (L_{0,0})=\Delta (L_{1,0})=0\) with Lemma \ref{prop4}, we see that \(\Delta\left(L_{n, 0}\right)=0\) for any \(n \in \mathbb{Z}\). This along with  Lemma \ref{prop3} gives
	\begin{align*}
	D_{L_{0,0}, x}&=\mathrm{ad}(\sum_{i \in \Z_{+}} a_{-i, i}(L_{0,0}, x) L_{-i, i})+ \lambda(L_{0,0}, x)d,\\
D_{L_{n, 0}, x}&=\mathrm{ad} (\lambda (L_{n, 0}, x) L_{0,0}+\sum_{i \in \Z_{+}} a_{ni+n-i, i}(L_{n, 0}, x) L_{ni+n-i, i})+\lambda(L_{n, 0}, x)d \,\,\,  {\rm for}\,\,\,  n\neq0,
		\end{align*}
where \(\lambda (L_{n, 0}, x ), a_{ni+n-i, i} (L_{n, 0}, x) \in \mathbb{C}\). Substituting these into $\Delta(x)=D_{L_{0,0}, x}(x)=D_{L_{n, 0}, x}(x)$ for $n\neq0$, one obtains
	\begin{align*}
\Delta(x)&=-\sum_{i\in \mathbb{Z}_{+}}\sum_{(\gamma,k) \in \mathbb{Z}\times\mathbb{Z}_{+}}(i+1)(k+\gamma) a_{-i, i}(L_{0,0}, x) \mu_{\gamma, k} L_{-i+\gamma, i+k}+\lambda (L_{0,0}, x) \sum_{(\gamma,k) \in \mathbb{Z}\times\mathbb{Z}_{+}} \gamma\mu_{\gamma, k} L_{\gamma, k}\\
&=\sum_{i \in \mathbb{Z}_{+}} \sum_{(\gamma,k) \in \mathbb{Z}\times\mathbb{Z}_{+}}(i+1)(nk+n-k-\gamma) a_{ni+n-i, i}(L_{n, 0}, x) \mu_{\gamma, k} L_{ni+n-i+\gamma, i+k}\\
		&\quad -\lambda(L_{n, 0}, x)\sum_{(\gamma,k) \in \mathbb{Z}\times\mathbb{Z}_{+}}k\mu_{\gamma, k} L_{\gamma, k}.
\end{align*}
 By taking  $n$ sufficiently large  in the second equality,  we  get \(\Delta(x)=\xi_{x} \sum_{(\gamma,k) \in \mathbb{Z}\times\mathbb{Z}_{+}} k\mu_{\gamma, k} L_{\gamma, k}\) with \(\xi_{x}=-\lambda (L_{n, 0}, x ) \). This completes the proof.
	\end{proof}
%	\begin{lemm}\label{prop6}
%		Let $\Delta$ be a 2-local derivation on $\mathcal{B}$ such that \(\Delta (L_{0,0})=\Delta (L_{1,0} )=\Delta (L_{-1,1} )=0\). Then \(\Delta (L_{p, 0}+L_{-2 p, 2 p} )=0\) for any  \(p \in \mathbb{N}\).
%	\end{lemm}
%	\begin{proof}
	\begin{lemm}\label{prop7}
		Let $\Delta$ be a 2-local derivation on $\mathcal{B}$ such that \(\Delta (L_{0,0})=\Delta (L_{1,0} )=\Delta (L_{-1,1})=0\). Then for any \(p \in \mathbb{N}\) and \(y \in \mathcal{B}\), we have
	\[
	D_{L_{p, 0}+L_{-2 p, 2 p}, y}=\mathrm{ad} (\sum_{\substack{\alpha+i \in p\mathbb{Z}_{+} \\(\alpha, i) \neq(0,0)}} a_{\alpha, i} (L_{p, 0}+L_{-2 p, 2 p}, y) L_{\alpha, i})
	\]
for some  $a_{\alpha, i} (L_{p, 0}+L_{-2 p, 2 p}, y) \in\C$.
	\end{lemm}
	\begin{proof}
For any  \(p \in \mathbb{N}\), in view of  \(\Delta (L_{0,0} )=\Delta (L_{1,0})=0\) and Lemma \ref{prop4}, we see that
		\begin{align}\label{8.1}
		\Delta (L_{p, 0}+L_{-2 p, 2 p} )=2 p \xi_{L_{p, 0}+L_{-2 p, 2 p}} L_{-2 p, 2p}
\end{align}
for some $\xi_{L_{p, 0}+L_{-2 p, 2 p}} \in\C$. Meanwhile, using  \(\Delta (L_{-1,1} )=0\) and Lemma \ref{prop3}, we can assume that $D_{L_{-1,1}, y}=\mathrm{ad} (\sum_{i \in \Z_{+}} a_{-i, i}(L_{-1,1}, y) L_{-i, i})$ with  \(a_{-i, i} (L_{-1,1}, y) \in \mathbb{C}\), which  gives
		\begin{align*}
			\Delta (L_{p, 0}+L_{-2 p, 2 p} ) & =D_{L_{-1,1}, L_{p, 0}+L_{-2 p, 2 p}} (L_{p, 0}+L_{-2 p, 2 p} ) \\
			& = \sum_{i \in \mathbb{Z}_{+}}[a_{-i, i}\left(L_{-1,1}, L_{p, 0}+L_{-2 p, 2 p}\right) L_{-i, i}, L_{p, 0}] \\
			& =-\sum_{i \in \mathbb{Z}_{+}}p(i+1) a_{-i, i} (L_{-1,1}, L_{p, 0}+L_{-2 p, 2 p}) L_{-i+p, i}.
		\end{align*}
This together with \eqref{8.1} forces 
\begin{align}\label{ad}
\Delta (L_{p, 0}+L_{-2 p, 2 p})=0. 		
\end{align} 
For any $y \in \mathcal{B}$, by Lemma \ref{prop2}, we can assume that
		\[
		D_{L_{p, 0}+L_{-2 p, 2 p}, y}=\mathrm{ad} (\sum_{(\alpha,i) \in \mathbb{Z}\times\mathbb{Z}_{+}} a_{\alpha, i} (L_{p, 0}+L_{-2 p, 2 p}, y ) L_{\alpha, i})+\lambda (L_{p, 0}+L_{-2 p, 2 p}, y ) d
		\]
for some \(a_{\alpha, i} (L_{p, 0}+L_{-2 p, 2 p}, y), \lambda (L_{p, 0}+L_{-2 p, 2 p}, y) \in \mathbb{C}\). Putting this and \eqref{ad} together gives
		\begin{align}	\label{ac}
			\Delta (L_{p, 0}+L_{-2 p, 2 p})= & D_{L_{p, 0}+L_{-2 p, 2 p}, y} (L_{p, 0}+L_{-2 p, 2 p}) \nonumber\\
			%= & { [\sum_{\alpha \in \mathbb{Z}, i \in \mathbb{Z}_{+}} a_{\alpha, i} (L_{p, 0}+L_{-2 p, 2 p}, y ) L_{\alpha, i}, L_{p, 0}+L_{-2 p, 2 p} ] } \nonumber\\
%			& +\lambda (L_{p, 0}+L_{-2 p, 2 p}, y ) D_{0} (L_{p, 0}+L_{-2 p, 2 p}) \nonumber\\
			= & \sum_{(\alpha,i) \in \mathbb{Z}\times\mathbb{Z}_{+}}(\alpha+i-p i-p) a_{\alpha, i} (L_{p, 0}+L_{-2 p, 2 p}, y) L_{\alpha+p, i} \nonumber\\
			& +\sum_{(\alpha,i) \in \mathbb{Z}\times\mathbb{Z}_{+}}(2 p+1)(\alpha+i) a_{\alpha, i} (L_{p, 0}+L_{-2 p, 2 p}, y) L_{\alpha-2 p, i+2 p} \nonumber\\
			& +p \lambda (L_{p, 0}+L_{-2 p, 2 p}, y) L_{p, 0}-2 p \lambda (L_{p, 0}+L_{-2 p, 2 p}, y) L_{-2 p, 2 p} \nonumber\\
			= & 0. 
		\end{align}
By comparing the coefficients of $L_{p, 0}, L_{-2 p, 2 p}$ and $L_{-(3n+2)p, 2(n+1)p}$ for $n\in\N$  in \eqref{ac}, we respectively have
\begin{eqnarray*}
&-p a_{0,0} (L_{p, 0}+L_{-2 p, 2 p}, y)+p \lambda (L_{p, 0}+L_{-2 p, 2 p}, y)=0,\\
&2p(p+1)a_{-3 p, 2 p}(L_{p, 0}+L_{-2 p, 2 p}, y)+2 p \lambda (L_{p, 0}+L_{-2 p, 2 p}, y)=0,\\
& p(2 np+2p+n+2) a_{-3(n+1) p, 2(n+1) p}(L_{p, 0}+L_{-2 p, 2 p}, y)\\
&+n p(2 p+1) a_{-3 n p, 2 n p}(L_{p, 0}+L_{-2 p, 2 p}, y)=0.
\end{eqnarray*}
These imply that $\lambda (L_{p, 0}+L_{-2 p, 2 p}, y)=a_{0,0} (L_{p, 0}+L_{-2 p, 2 p}, y)=0$, for otherwise we would have $ a_{-3 n p, 2 n p}(L_{p, 0}+L_{-2 p, 2 p}, y)\neq0$ for all $n\in\Z_+$,  which  contradicts the finiteness of the set  $\{(\alpha,i)\mid a_{\alpha, i}\neq0\}$. Now for $(\alpha,i)\in \mathbb{Z}\times\mathbb{Z}_{+}\setminus\{(0,0)\}$ and $n\in\Z_+$, by observing the coefficients of \(L_{\alpha-(3n+2)p, i+2(n+1)p}\) in \eqref{ac}, we get the following recurrence relation
\begin{eqnarray*}
&(\alpha+i-p(n+i+2+2pn+2p))a_{\alpha-3(n+1)p, i+2(n+1)p}(L_{p, 0}+L_{-2 p, 2 p}, y)\\
&=-(2 p+1)(\alpha+i-n p) a_{\alpha-3 n p, i+2 n p}(L_{p, 0}+L_{-2 p, 2 p}, y).
\end{eqnarray*}
If \(\alpha+i \notin p \mathbb{Z}_{+}\), then we must have  \(a_{\alpha, i} (L_{p, 0}+L_{-2 p, 2 p}, y)=0\), as otherwise $ a_{\alpha-3 n p, i+2 n p}(L_{p, 0}+L_{-2 p, 2 p}, y)\neq0$ for all $n\in\Z_+$, which is absurd. This completes the proof. 
	\end{proof}

\begin{lemm}\label{prop8}
	Let $\Delta$ be a 2-local derivation on $\mathcal{B}$ such that \(\Delta (L_{0,0})=\Delta (L_{1,0} )=\Delta (L_{-1,1})=0\). Then $\d(x)=0$ for all $x\in\mathcal{B}$.
\end{lemm}
\begin{proof}
Take an arbitrary but fixed \(x=\sum_{(\gamma,k) \in \mathbb{Z}\times \mathbb{Z}_{+}} \mu_{\gamma, k} L_{\gamma, k} \in \mathcal{B}\) with \( \mu_{\gamma, k} \in \mathbb{C}\). It follows from  \(\Delta (L_{0,0} )=\Delta (L_{1,0} )=\Delta (L_{-1,1})=0\), Lemmas \ref{prop5} and \ref{prop7} that
	\begin{align}
	\label{gh}\Delta(x)&=\xi_{x} \sum_{(\gamma,k) \in \mathbb{Z}\times \mathbb{Z}_{+}} k\mu_{\gamma, k} L_{\gamma, k},\\
   D_{L_{p, 0}+L_{-2 p, 2 p}, x}&=\mathrm{ad}(\sum_{\substack{\alpha+i \in p \mathbb{Z}_{+}\\(\alpha, i) \neq(0,0)}} a_{\alpha, i} (L_{p, 0}+L_{-2 p, 2 p}, x) L_{\alpha, i} ),\,\,\,\,p\in\N,\nonumber 
	\end{align}
where $\xi_{x}, a_{\alpha, i} (L_{p, 0}+L_{-2 p, 2 p}, x)\in\C$. Now we compute
	\begin{align}\label{ef}
		\Delta(x)= & D_{L_{p, 0}+L_{-2 p, 2 p}, x}(x)\nonumber \\
		= & \sum_{\substack{\alpha+i=0 \\
					(\alpha, i) \neq(0,0)}}\sum_{(\gamma,k) \in \mathbb{Z}\times\mathbb{Z}_{+}} [a_{\alpha, i} (L_{p, 0}+L_{-2 p, 2 p}, x) L_{\alpha, i},  \mu_{\gamma, k} L_{\gamma, k}] \nonumber \\
		& + \sum_{\alpha+i \in p \mathbb{N}}\sum_{(\gamma,k) \in \mathbb{Z}\times\mathbb{Z}_{+}}[ a_{\alpha, i} (L_{p, 0}+L_{-2 p, 2 p}, x) L_{\alpha, i}, \mu_{\gamma, k} L_{\gamma, k}] \nonumber \\
		=& -\sum_{i \in \mathbb{N}} \sum_{(\gamma,k) \in \mathbb{Z}\times\mathbb{Z}_{+}}(i+1)(\gamma+k) a_{-i, i} (L_{p, 0}+L_{-2 p, 2 p}, x) \mu_{\gamma, k} L_{-i+\gamma, i+k} \nonumber\\
		& +\sum_{\alpha+i \in p \mathbb{N}} \sum_{ (\gamma,k) \in \mathbb{Z}\times\mathbb{Z}_{+}}((\alpha-1)(k+1)-(\gamma-1)(i+1)) a_{\alpha, i} (L_{p, 0}+L_{-2 p, 2 p}, x ) L_{\alpha+\gamma, i+k}.   
	\end{align} 
If \(\alpha+i \in p \mathbb{N}\), then by letting $p$ be large enough and equating the coefficients of \(L_{\alpha+\gamma, i+k}\) in \eqref{gh} and \eqref{ef}, we get  $((\alpha-1)(k+1)-(\gamma-1)(i+1)) a_{\alpha, i} (L_{p, 0}+L_{-2 p, 2 p}, x)=0$.  Now \eqref{gh} and \eqref{ef} become
	\begin{align}\label{vb}
\Delta(x)=&\xi_{x} \sum_{(\gamma,k) \in \mathbb{Z}\times \mathbb{Z}_{+}} k\mu_{\gamma, k} L_{\gamma, k}\nonumber\\
=&-\sum_{i \in \mathbb{N}}\sum_{ (\gamma,k) \in \mathbb{Z}\times \mathbb{Z}_{+}}(i+1)(\gamma+k) a_{-i, i} (L_{p, 0}+L_{-2 p, 2 p}, x) \mu_{\gamma, k} L_{-i+\gamma, i+k}.
	\end{align}
In the case \(\xi_{x}=0\), we immediately get  \(\Delta(x)=0\) from the first equality of \eqref{vb}. Next consider the case \(\xi_{x}\neq 0\). If $\mu_{\gamma,k}=0$ for all $\gamma\neq-k$, then  \(\Delta(x)=0\) by the second equality of \eqref{vb}. Assume now that $\mu_{\gamma_{0},k_{0}}\neq0$ for some $\gamma_{0}\neq-k_{0}$. For any $i\in\N$, by comparing the coefficients of  $L_{-i+\gamma_{0}, i+k_{0}}$ in \eqref{vb}, we get
\begin{align}\label{nm2}
%\xi_{x}k\mu_{-k,k}&=0,\\
\xi_{x} (k_0+i) \mu_{\gamma_{0}-i, k_{0}+i}&=-(i+1 )(\gamma_{0}+k_{0}) a_{-i, i} (L_{p, 0}+L_{-2 p, 2 p}, x) \mu_{\gamma_{0}, k_{0}}.
\end{align} 
Therefore,  $a_{-i, i} (L_{p, 0}+L_{-2 p, 2 p}, x )=0$ for any $i\in\N$. If not, suppose  that $a_{-i_{0},i_{0}}(L_{p, 0}+L_{-2 p, 2 p}, x)\neq0$ for some $i_{0}\in\N$. Then from  \eqref{nm2}, we infer that  $\mu_{\gamma_{0}-ni_{0}, k_{0}+ni_{0} }\neq0$ for all $n\in\Z_+$, contradicting the fact that the set $\{(\gamma,k)\mid \mu_{\gamma,k}\neq0\}$ is finite. Consequently, we  get \(\Delta(x)=0\) by the second equality of \eqref{vb}.  So the claim is proved.
\end{proof}
Now we can formulate our main result in this section. 

\begin{theo}\label{nm}
	Every  2-local derivation on the Block-type Lie algebra $\mathcal{B}$ is a derivation.
\end{theo}
\begin{proof}
	Let $\d$ be a 2-local derivation on $\mathcal{B}$. Take a derivation \(D_{L_{0,0}, L_{1,0}}\) such that
	\[
	\Delta (L_{0,0})=D_{L_{0,0}, L_{1,0}} (L_{0,0}) \quad \mathrm{and} \quad \Delta (L_{1,0} )=D_{L_{0,0}, L_{1,0}} (L_{1,0}).
	\]
Set \(\Delta_{1}=\Delta-D_{L_{0,0}, L_{1,0}}\). Then $\Delta_{1}$ is a 2-local derivation such that \(\Delta_{1} (L_{0,0} )=\Delta_{1} (L_{1,0} )=0\). It
follows from  Lemma \ref{prop4} that \(\Delta_{1} (L_{-1,1} )=\xi_{L_{-1,1}} L_{-1,1}\) for some \(\xi_{L_{-1,1}} \in \mathbb{C}\). Now we set \(\Delta_{2}=\Delta_{1}+\xi_{L_{-1,1}} \mathrm{ad} (L_{0,0} )+\xi_{L_{-1,1}}d\), then \(\Delta_{2}\) is a 2-local derivation such that
\begin{align*}
	\Delta_{2} (L_{0,0} )&=\Delta_{1}(L_{0,0})+\xi_{L_{-1,1}} [L_{0,0}, L_{0,0}]+\xi_{L_{-1,1}} d (L_{0,0})=0,\\
	\Delta_{2} (L_{1,0} )&=\Delta_{1} (L_{1,0})+\xi_{L_{-1,1}} [L_{0,0}, L_{1,0}]+\xi_{L_{-1,1}} d(L_{1,0})\\
&=-\xi_{L_{-1,1}} L_{1,0}+\xi_{L_{-1,1}} L_{1,0}=0,\\
	\Delta_{2} (L_{-1,1} )&=\Delta_{1} (L_{-1,1})+\xi_{L_{-1,1}} [L_{0,0}, L_{-1,1}]+\xi_{L_{-1,1}}d(L_{-1,1} )\\
&=\xi_{L_{-1,1}} L_{-1,1}-\xi_{L_{-1,1}} L_{-1,1}=0.
\end{align*}
By Lemma \ref{prop8}, we see that $\Delta_{2}=\Delta_{1}+\xi_{L_{-1,1}} \mathrm{ad} (L_{0,0})+\xi_{L_{-1,1}}d \equiv 0$. Therefore, \(\Delta=D_{L_{0,0}, L_{1,0}}-\xi_{L_{-1,1}} \mathrm{ad}(L_{0,0})-\xi_{L_{-1,1}}d\) is a derivation, completing the proof.
\end{proof}
	
	%\subsection*{Acknowledgements}
	%The authors would like to thank the referee for his/her helpful comments and suggestion.

\end{document}